\def\benm{\begin{enumerate}}
\def\eenm{\end{enumerate}}

\def\bal{\begin{align}}
\def\eal{\end{align}}

\documentclass{amsart}
\usepackage{amsmath, amsthm, amscd, amsfonts,amssymb, graphicx}

\newtheorem{theorem}{Theorem}[section]

\theoremstyle{definition}

\newtheorem{example}[theorem]{Example}

\newtheorem{proposition}[theorem]{Proposition}
\newtheorem{corollary}[theorem]{Corollary}

\theoremstyle{remark}

\numberwithin{equation}{section}

\begin{document}

\title[]{ZAK TRANSFORM FOR SEMIDIRECT PRODUCT OF LOCALLY COMPACT GROUPS}

\author[Ali Akbar Arefijamaal and Arash Ghaani Farashahi]{Ali Akbar Arefijamaal and Arash Ghaani Farashahi$^{*}$}

\address{
Ali Akbar Arefijamaal, Department of Mathematics, Sabzevar Tarbiat
Moallem University, Sabzevar, Iran.}
\email{arefijamal@sttu.ac.ir}
\address{Department of Pure Mathematics, Faculty of Mathematical sciences, Ferdowsi University of
Mashhad (FUM), P. O. Box 1159, Mashhad 91775, Iran.}
\email{ghaanifarashahi@hotmail.com}

\curraddr{}





\date{}


\thanks{$^*$Corresponding author}
\thanks{Emails: Ali Akbar Arefijamaal (arefijamal@sttu.ac.ir),  Arash Ghaani Farashahi (ghaanifarashahi@hotmail.com).}

\begin{abstract}
Let $H$ be a locally compact group and $K$ be an LCA group also let $\tau:H\to Aut(K)$ be a continuous homomorphism and $G_\tau=H\ltimes_\tau K$ be the semidirect product of $H$ and $K$ with respect to $\tau$. In this article we define the Zak transform $\mathcal{Z}_L$ on $L^2(G_\tau)$ with respect to a $\tau$-invariant uniform lattice $L$ of $K$ and we also show that the Zak transform satisfies the Plancherel formula. As an application we show that how these techniques apply for the semidirect product group $\mathrm{SL}(2,\mathbb{Z})\ltimes_\tau\mathbb{R}^2$ and also the Weyl-Heisenberg groups.
\end{abstract}

\maketitle

\section{\bf{Introduction}}

The Zak transform has been used successfully in various
applications in physics, such as studying coherent
states representations in quantum field theory \cite{Jan82,Kl85}, decomposing the Hamiltonian
and rewrite it as magnetic Weyl
quantization of an operator-valued function \cite{NIT} and also as discrete Fourier transform performed on the signal blocks.
It can be also considered as the polyphase representation of periodic
signals and an important tool in the analysis of Gabor system and
signal theory \cite{Jan88}. Furthermore, the Zak transforms of the
B-splines, appear as a base for the integral representation of
discrete splines. This integral representation is similar to the
Fourier integral, and the Zak splines play the part of the Fourier
exponentials \cite{Pev08}.

The Zak transform on $\mathbb{R}$ was introduced in 1950 by
Gelfand and it was rediscovered in quantum mechanic representation
by Zak \cite{Zak67}. Let $f\in L^{2}(\mathbb{R})$, the function
$Zf:\mathbb{R}\times\mathbb{R}\rightarrow\mathbb{C}$ defined by
\begin{eqnarray}
Zf(x,y)=\sum_{k=-\infty}^{+\infty}f(x+k)e^{2\pi iyk},
\end{eqnarray}
is called the Zak transform of $f$.
The Zak transform on LCA (locally compact abelian) groups was introduced by Weil
\cite{Weil64} and developed by many authors
\cite{Gro98,Kut03,Kan98}. Recently, applications of the Zak transform for LCA groups $\mathbb{T}$
and $\mathbb{Z}$ has been rediscovered in mathematical physics to study representations of finite quantum systems (see \cite{Zhang}).
An approach to the Zak transform on certain non-abelian locally
compact groups can be found in \cite{Kut02}. Many groups which appear in mathematical physic and quantum mechanic are non-abelian although they can be consider as semi-direct product of locally compact groups.

In this paper which contains 4 sections, we introduce an approach to define the Zak transform on semidirect product groups of the form $G_\tau=H\ltimes_\tau K$ where $K$ is an LCA group and $\tau:H\to Aut(K)$ is a continuous homomorphism.
This article organized as follows; section 2 is devoted to fix notations containing a summary on the standard harmonic analysis of semi-direct product groups and also an overview of the Zak transform on LCA groups.
In order to define the Zak transform on $ G_\tau$, we first introduce the continuous homomorphisms $\widehat{\tau}:H\to Aut(\widehat{K})$ and $\tau^L:H\to Aut(K/L)$, for a $\tau$-invariant closed subgroup $L$ of $K$ in section 3. Then we construct the semidirect products $H\ltimes_\tau\widehat{K}$ and $H\ltimes_\tau K/L$ and study the basic properties of them. Applying these continuous homomorphisms for a given $\tau$-invariant closed subgroup $L$ of $K$, we define the continuous homomorphism $\tau^{\times,L}=\tau^L\times\widehat{\tau}^{L^\perp}:H\to Aut(K/L\times\widehat{K}/L^\perp)$ which induces the locally compact group $G_{\tau^{\times,L}}=H\ltimes_{\tau^{\times,L}}\left(K/L\times\widehat{K}/L^\perp\right)$. The $\tau$-Zak transform $\mathcal{Z}_Lf$ of a function $f\in L^2(G_\tau)$ with respect to a $\tau$-invariant uniform lattice $L$ of $K$ is defined as a function on the semidirect product $G_{\tau^{\times,L}}$. It is also proved that $\mathcal{Z}_L:L^2(G_{\tau})\to L^2(G_{\tau^{\times,L}})$ is an isometric transform.

Finally, we study these methods for the semidirect group $\mathrm{SL}(2,\mathbb{Z})\ltimes_\tau\mathbb{R}^2$ and also the Weyl-Heisenberg groups.

%

\section{\bf{Preliminaries and notations}}

Let $H$ and $K$ be locally compact groups with left Haar measures $dh$ and $dk$ respectively,
also let $\tau:H\to Aut(K)$ be a homomorphism such that the map $(h,k)\mapsto \tau_h(k)$ is continuous from $H\times K$ onto $K$. For simplicity in notation we often use $k^h$ instead of $\tau_h(k)$ for all $h\in H$ and $k\in K$.

There is also a natural topology, sometimes called
Braconnier topology, turning $Aut(K)$ into a Hausdorff topological group (not necessarily locally compact),
which is defined by the sub-base of identity neighbourhoods
\begin{equation}
\mathcal{B}(F,U)= \{ \alpha\in Aut(K): \alpha(k),\alpha^{-1}(k)\in Uk\ \forall k\in F\},
\end{equation}
where $F\subseteq K$ is a compact set and $U\subseteq K$ is an identity neighbourhood. The continuity of a homomorphism $\tau:H\to Aut(K)$ is equivalent to the continuity of the map $(h,k)\mapsto \tau_h(k)$ from $H\times K$ onto $K$ (see \cite{HO}).

The semidirect product $G_\tau=H\ltimes_\tau K$ is the locally compact topological group with the underlying set $H\times K$ which is equipped by the product topology and also the group operations are given by
\begin{equation}\label{0.1}
(h,k)\ltimes_\tau(h',k'):=(hh',k\tau_h(k'))\hspace{0.5cm}{\rm and}\hspace{0.5cm}(h,k)^{-1}:=(h^{-1},\tau_{h^{-1}}(k^{-1})).
\end{equation}
The left Haar measure of $G_\tau$ is $d\mu_{G_\tau}(h,k)=\delta_K(h)dhdk$ and the modular function of $G_\tau$ is $\Delta_{G_\tau}(h,k)=\delta_K(h)\Delta_H(h)\Delta_K(k)$, where the positive and continuous homomorphism $\delta_K:H\to(0,\infty)$ is given by (15.29 of \cite{HR1})
\begin{equation}\label{t}
dk=\delta_K(h)d(\tau_h(k)).
\end{equation}
If $L$ is a closed subgroup of $K$ which is $\tau$-invariant (i.e. $\tau_h(L)\subseteq L$ for all $h\in H$) with the left Haar measure $dl$, then $\tau:H\to Aut(L)$ is a well-defined continuous homomorphism and $H\ltimes_{\tau}L$ is a locally compact group with the left Haar measure $\delta_L(h)dhdl$, where $\delta_L:H\to(0,\infty)$ is given by
\begin{equation}\label{tt}
dl=\delta_L(h)d(\tau_h(l)).
\end{equation}
It is clear that if $H$ is a compact group we have $\delta_L=\delta_K=1$ and also if $L$ is an open subgroup of $K$ we get $\delta_L=\delta_K$.
From now on, for all $p\ge 1$ we denote by $L^p(G_\tau)$ the Banach space $L^p(G_\tau,\mu_{G_\tau})$ and also $L^p(K)$ stands for $L^p(K,dk)$.
When $f\in L^p(G_\tau)$, for a.e. $h\in H$ the function $f_h$ defined on $K$ via $f_h(k):=f(h,k)$ belongs to $L^p(K)$ (see \cite{FollR}).

If $L$ is a closed subgroup of an LCA group $K$ then $K/L$ is a locally compact group with left Haar measure $\sigma_{K/L}$ and also
$L^1(K/L,\sigma_{K/L})$ is precisely the set of all functions of the form $T_Lv$ with $v\in L^1(K)$ and
\begin{equation}
T_Lv(k+L)=\int_Lf(k+l)dl.
\end{equation}
In fact, $T_L:L^1(K)\to L^1(K/L,\sigma_{K/L})$ given by $v\mapsto T_Lv$ is a surjective bounded linear map. As well as, all $v\in L^1(K)$ satisfy the following Weil's formula (see \cite{FollH});
\begin{equation}
\int_Kv(k)dk=\int_{K/L}T_Lv(k+L)d\sigma_{K/L}(k+L).
\end{equation}
If $K$ is an LCA group all irreducible representations of $K$ are one-dimensional. Thus, if $(\pi,\mathcal{H}_\pi)$ is an irreducible unitary representation of $K$ we have $\mathcal{H}_\pi=\mathbb{C}$ and also according to the Shur's Lemma there exists a continuous homomorphism $\omega$ of $K$ into the circle group $\mathbb{T}$ such that for each $k\in K$ and $z\in\mathbb{C}$ we have $\pi(k)(z)=\omega(k)z$. Such homomorphisms are called characters of $K$ and the set of all characters of $K$ denoted by $\widehat{K}$. It is a usual notation to use $\langle k,\omega\rangle$ instead of $\omega(k)$. If $\widehat{K}$ equipped by the topology of compact convergence on $K$ which coincides with the $w^*$-topology that $\widehat{K}$ inherits as a subset of $L^\infty(K)$, then $\widehat{K}$ with respect to the dot product of characters is an LCA group which is called the dual group of $K$.
The linear map $\mathcal{F}_K:L^1(K)\to \mathcal{C}(\widehat{K})$ defined by $v\mapsto \mathcal{F}_K(v)$ via
\begin{equation}\label{SF}
\mathcal{F}_K(v)(\omega)=\widehat{v}(\omega)=\int_Kv(k)\overline{\omega(k)}dk,
\end{equation}
is called the Fourier transform on $K$, where $\mathcal{C}(K)$ is the set of all continuous functions on $K$. 
The Fourier transform (\ref{SF}) on $L^1(K)\cap L^2(K)$ is an isometric transform and it extends uniquely to a unitary isomorphism from $L^2(K)$ to $L^2(\widehat{K})$ (Theorem 4.25 of \cite{FollH}) also each $v\in L^1(K)$ with $\widehat{v}\in L^1(\widehat{K})$ satisfies the following Fourier inversion formula (Theorem 4.32 of \cite{FollH});
\begin{equation}
v(k)=\int_{\widehat{K}}\widehat{v}(\omega)\omega(k)d\omega\ {\rm for} \ {\rm a.e.}\ k\in K.
\end{equation}
If $L$ is a closed subgroup of an LCA group $K$ the
annihilator of $L$ in $K$ is defined by
\begin{equation}
L^{\perp}=\{\omega\in\widehat{K}:\ \omega(l)=1~ \text{for all}~ l\in
L \},
\end{equation}
which is a closed subgroup of $\widehat{K}$. Then, $(L^{\perp})^{\perp}=L$ and
$\widehat{K}/L^{\perp}=\widehat{L}$ also $\widehat{K/L}=L^{\perp}$
(see Theorem 4.39 of \cite{FollH}). By a uniform lattice
$L$ of $K$ we mean a discrete and co-compact (i.e $K/L$ is compact) subgroup $L$ of $K$. If $K$ is second countable it is always guaranteed that $K$ possesses a uniform lattice (see \cite{Kan98}). 

The Zak transform associated to a uniform lattice $L$ of
$v\in \mathcal{C}_c(K)$ is defined on $K\times \widehat{K}$ by
\begin{eqnarray}
Z_Lv(k,\omega)=\sum_{l\in L}f(k+l)\omega(l),
\end{eqnarray}
where $\mathcal{C}_c(K)$ stands for the function space of all continuous functions on $K$ with compact support.
It is shown that $Z_L:\mathcal{C}_c(K)\to\mathcal{C}_c(K/L\times\widehat{K}/L^\perp)$ is an isometry in $L^2$-norms and so that it can be uniquely extended into the Zak transform $Z_L:L^2(K)\rightarrow
L^2(K/L\times\widehat{K}/L^\perp)$ which is still an isometry (see \cite{Gro98, Kut03}).


\section{{\bf $\tau$-Zak transform}}
Throughout this article, let $H$ be locally compact group, K an LCA group and $\tau:H\to Aut(K)$ be a continuous homomorphism also let $G_\tau=H\ltimes_\tau L$.
Define $\widehat{\tau}:H\to Aut(\widehat{K})$ via $h\mapsto \widehat{\tau}_h$, given by
\begin{equation}\label{AAA}
\widehat{\tau}_h(\omega):=\omega_h=\omega\circ\tau_{h^{-1}}
\end{equation}
for all $\omega \in \widehat{K}$, where $\omega_h(k)=\omega(\tau_{h^{-1}}(k))$ for all $k\in K$.  
According to (\ref{AAA}) for all $h\in H$ we have $\widehat{\tau}_h\in Aut(\widehat{K})$ and also $h\mapsto \widehat{\tau}_h$ is a homomorphism from $H$ into $Aut(\widehat{K})$. Because if $h,t\in H$ then for all $\omega\in\widehat{K}$ and also $k\in K$ we have
\begin{align*}
\widehat{\tau}_{th}(\omega)(k)&=\omega_{th}(k)
\\&=\omega(\tau_{(th)^{-1}}(k))
\\&=\omega(\tau_{h^{-1}}\tau_{t^{-1}}(k))
\\&=\omega_h(\tau_{t^{-1}}(k))
=\widehat{\tau}_h(\omega)(\tau_{t^{-1}}(k))=\widehat{\tau}_t[\widehat{\tau}_h(\omega)](k).
\end{align*}
Thus, we can prove the following theorem.
\begin{theorem}\label{T}
Let $H$ be a locally compact group and $K$ be an LCA group also $\tau:H\to Aut(K)$ be a continuous homomorphism and let $\delta_K:H\to(0,\infty)$ be the positive continuous homomorphism satisfying (\ref{t}). The semidirect product $G_{\widehat{\tau}}=H\ltimes_{\widehat{\tau}}\widehat{K}$ is a locally compact group with the left Haar measure $d\mu_{G_{\widehat{\tau}}}(h,\omega)=\delta_K(h)^{-1}dhd\omega$.
\end{theorem}
\begin{proof}
Continuity of the homomorphism $\widehat{\tau}:H\to Aut(\widehat{K})$ given in (\ref{AAA}) guaranteed by Theorem 26.9 of \cite{HR1}.
Hence, the semidirect product $G_{\widehat{\tau}}=H\ltimes_{\widehat{\tau}}\widehat{K}$ is a locally compact group.
We also claim that the Plancherel measure $d\omega$ on $\widehat{K}$ for all $h\in H$ satisfies
\begin{equation}\label{AA2}
d\omega_h=\delta_K(h)d\omega.
\end{equation}
Let $h\in H$ and also $v\in L^1(K)$. Using (\ref{t}) we have $v\circ \tau_h\in L^1(K)$ with $\|v\circ \tau_h\|_{L^1(K)}=\delta_K(h)\|v\|_{L^1(K)}$.
Thus, for all $\omega\in\widehat{K}$ we achieve
\begin{align*}
\widehat{v\circ\tau_h}(\omega)
&=\int_Kv\circ\tau_h(k)\overline{\omega(k)}dk
\\&=\int_Kv(k^h)\overline{\omega(k)}dk
\\&=\int_Kv(k)\overline{\omega_h(k)}dk^{h^{-1}}
=\delta_K(h)\int_Kv(k)\overline{\omega_h(k)}dk=\delta_K(h)\widehat{v}(\omega_h).
\end{align*}
Now let $v\in L^1(K)\cap L^2(K)$. Due to the Plancherel theorem (Theorem 4.25 of \cite{FollH}) and also preceding calculation, for all $h\in H$ we get
\begin{align*}\label{}
\int_{\widehat{K}}|\widehat{v}(\omega)|^2d\omega_h
&=\int_{\widehat{K}}|\widehat{v}(\omega_{h^{-1}})|^2d\omega
\\&=\delta_K(h)^2\int_{\widehat{K}}|\widehat{v\circ\tau_{h^{-1}}}(\omega)|^2d\omega
\\&=\delta_K(h)^2\int_{{K}}|{v\circ\tau_{h^{-1}}}(k)|^2dk
\\&=\delta_K(h)^2\int_{{K}}|v(k)|^2dk^h
=\delta_K(h)\int_{{K}}|v(k)|^2dk=\int_{\widehat{K}}|\widehat{v}(\omega)|^2\delta_K(h)d\omega,
\end{align*}
which implies (\ref{AA2}). Therefore, $d\mu_{G_{\widehat{\tau}}}(h,\omega)=\delta_K(h)^{-1}dhd\omega$ is a left Haar measure for $G_{\widehat{\tau}}=H\ltimes_{\widehat{\tau}}\widehat{K}$.
\end{proof}
If $L$ is a closed $\tau$-invariant subgroup of $K$, let $\tau^L:H\to Aut(K/L)$ via $h\mapsto \tau^L_h$ be given by
\begin{equation}\label{S1}
\tau^L_h(k+L):=\tau_h(k)+L=k^h+L.
\end{equation}
For all $h\in H$ we have $\tau^L_h\in Aut(K/L)$ and $\tau^L:H\to Aut(K/L)$ is a well-defined homomorphism.
For simplicity in notations, from now on we use $(k+L)^h$ instead of $\tau^L_h(k+L)$. Thus, via algebraic structures we can consider the semidirect product $G_{\tau^L}=H\ltimes_{\tau^L} K/L$. The group operation for $(h,k+L),(h',k'+L)\in G_{\tau^L}$ is
\begin{equation}
(h,k+L)\ltimes_{\tau^L}(h',k'+L)=(hh',k+k'^h+L).
\end{equation}
In the next theorem we show that $G_{\tau^L}=H\ltimes_{\tau^L} K/L$ is a locally compact group and we also identify the left Haar measure.
\begin{theorem}\label{KL}
Let $H$ be a locally compact group, $K$ be an LCA group and $\tau:H\to Aut(K)$ be a continuous homomorphism
also let $L$ be a closed $\tau$-invariant subgroup of $K$ and $\delta_K,\delta_L:H\to (0,\infty)$ be the continuous homomorphisms satisfying (\ref{t}) and (\ref{tt}) respectively. The semidirect product $G_{\tau^L}=H\ltimes_{\tau^L} K/L$ is a locally compact group with the left Haar measure $d\mu_{G_{\tau^L}}(h,k+L)=\delta_K(h)\delta_L(h^{-1})dhd\sigma_{K/L}(k+L)$, where $\sigma_{K/L}$ is the left Haar measure of $K/L$.
\end{theorem}
\begin{proof}
Continuity of $\tau^L:H\to Aut(K/L)$ is an immediate consequence of continuity of $\tau:H\to Aut(K)$. Hence, the semidirect product $G_{\tau^L}=H\ltimes_{\tau^L} K/L$ is a locally compact group. Let $h\in H$ and $v\in L^1(K)$ be given. Then, $v\circ\tau_h\in L^1(K)$ and also for $k+L\in K/L$ we have
\begin{align*}
T_L(v\circ \tau_h)(k+L)&=\int_Lv\circ\tau_h(k+l)dl
\\&=\int_Lv(k^h+l^h)dl
\\&=\int_Lv(k^h+l)dl^{h^{-1}}
=\delta_L(h)\int_Lv(k^h+l)dl=\delta_L(h)T_L(v)(k^h+L).
\end{align*}
Now, due to the Weil's formula we get
\begin{align*}
\int_{K/L}T_L(v)(k+L)d\sigma_{K/L}\left((k+L)^h\right)&=
\int_{K/L}T_L(v)(k+L)d\sigma_{K/L}\left(k^h+L\right)
\\&=\int_{K/L}T_L(v)(k^{h^{-1}}+L)d\sigma_{K/L}\left(k+L\right)
\\&=\delta_L(h)\int_{K/L}T_L(v\circ\tau_{h^{-1}})(k+L)d\sigma_{K/L}\left(k+L\right)
\\&=\delta_L(h)\int_Kv\circ\tau_{h^{-1}}(k)dk
\\&=\delta_L(h)\int_Kv(k)dk^h
\\&=\delta_L(h)\delta_K(h^{-1})\int_Kv(k)dk=\int_{K/L}T_L(v)(k+L)d\sigma_{K/L}(k+L).
\end{align*}
Thus, $d\mu_{G{\tau^L}}(h,k+L)=\delta_K(h)\delta_L(h^{-1})dhd\sigma_{K/L}(k+L)$ is the left Haar measure of $G_{\tau^L}=H\ltimes_{\tau^L} K/L$.
\end{proof}
Applying Theorem \ref{KL} for $\widehat{K}$ and $L^\perp$ and also using Theorem \ref{T} we achieve the following corollary.
\begin{corollary}\label{TKL}
{\it Let $H$ be a locally compact group, $K$ be an LCA group and $\tau:H\to Aut(K)$ be a continuous homomorphism
also let $L$ be a closed $\tau$-invariant subgroup of $K$ and $\delta_K,\delta_L:H\to (0,\infty)$ be the continuous homomorphisms satisfying (\ref{t}) and (\ref{tt}) respectively. The semidirect product $G_{\widehat{\tau}^{L^\perp}}=H\ltimes_{\widehat{\tau}^{L^\perp}}\widehat{K}/L^\perp$ is a locally compact group with the left Haar measure $d\mu_{\widehat{\tau}^{L^\perp}}(h,\omega L^\perp)=\delta_K(h^{-1})\delta_L(h)dhd\sigma_{\widehat{K}/L^\perp}(\omega L^\perp)$, where $\sigma_{\widehat{K}/L^\perp}$ is the left Haar measure of $\widehat{K}/L^\perp$.}
\end{corollary}
We can also conclude the following propositions.
\begin{proposition}
{\it Let $H$ be a locally compact group and $K$ be an LCA group also $\tau:H\to Aut(K)$ be a continuous homomorphism
and let $L$ be a closed $\tau$-invariant subgroup of $K$. The semidirect product group $G_{\widehat{\tau^L}}$ is a locally compact group and also $\Psi:G_{\widehat{\tau^L}}\to H\ltimes_{\widehat{\tau}}L^\perp$ given by $(h,\zeta)\mapsto\Psi(h,\zeta)=(h,[\zeta])$ where $[\zeta](k)=\zeta(k+L)$, is a topological group isomorphism.}
\end{proposition}
\begin{proof}
For all $\zeta\in \widehat{K/L}$, $h\in H$ and also $k\in K$ we have
\begin{align*}
[\zeta_h](k)&=\zeta_h(k+L)
\\&=\zeta\left((k+L)^{h^{-1}}\right)
=\zeta(k^{h^{-1}}+L)=[\zeta]_h.
\end{align*}
Let $(h,\zeta),(h',\zeta')$ in $G_{\widehat{\tau^L}}$. Due to Proposition 4.38 of \cite{FollH} we get
\begin{align*}
\Psi[(h,\zeta)\ltimes_{\widehat{\tau^L}}(h',\zeta')]&=\Psi(hh',\zeta\zeta'_h)
\\&=(hh',[\zeta\zeta'_h])
\\&=(hh',[\zeta][\zeta'_h])
\\&=(hh',[\zeta][\zeta']_h)=\Psi(h,\zeta)\ltimes_{\widehat{\tau}^{L^\perp}}\Psi(h',\zeta').
\end{align*}
Now, again Proposition 4.38 of \cite{FollH} implies that $\Psi$ is a topological group isomorphism.
\end{proof}
Also, by similar argument we achieve the following proposition.
\begin{proposition}
{\it Let $H$ be a locally compact group and $K$ be an LCA group also $\tau:H\to Aut(K)$ be a continuous homomorphism
and let $L$ be a closed $\tau$-invariant subgroup of $K$. The semidirect product group $G_{\widehat{\tau}^{L^\perp}}$ is a locally compact group and also $\Phi:G_{\widehat{\tau}^{L^\perp}}\to H\ltimes_{\widehat{\tau}}\widehat{L}$ given by $(h,\omega L^\perp)\mapsto(h,\omega|_L)$, is a topological group isomorphism.}
\end{proposition}
If $L$ is a $\tau$-invariant closed subgroup of $K$. Define $\tau^{\times,L}:=\tau^L\times\widehat{\tau}^{L^\perp}:H\to Aut(K/L\times\widehat{K}/L^\perp)$ via $h\mapsto \tau^{\times,L}_h$ given by
\begin{equation}
\tau^{\times,L}_h(k+L,\omega L^\perp)=(k+L,\omega L^\perp)^h:=(k^h+L,\omega_hL^\perp).
\end{equation}
It can be easily checked that $\tau^{\times,L}=\tau^L\times\widehat{\tau}^{L^\perp}:H\to Aut(K/L\times\widehat{K}/L^\perp)$ is a continuous homomorphism. Thus, we can prove the following theorem.
\begin{theorem}\label{ZG}
Let $H$ be a locally compact group, $K$ be an LCA group and $\tau:H\to Aut(K)$ be a continuous homomorphism
also let $L$ be a closed $\tau$-invariant subgroup of $K$ and $\delta_K,\delta_L:H\to (0,\infty)$ be the continuous homomorphisms satisfying (\ref{t}) and (\ref{tt}) respectively. The semidirect product $G_{\tau^{\times,L}}=H\ltimes_{\tau^{\times,L}}\left(K/L\times\widehat{K}/L^\perp\right)$ is a locally compact group with the left Haar measure $d\mu_{G_{\tau^{\times,L}}}(h,k+L,\omega L^\perp)=dhd\sigma_{K/L}(k+L)d\sigma_{\widehat{K}/L^\perp}(\omega L^\perp)$, where $\sigma_{K/L}$ and $\sigma_{\widehat{K}/L^\perp}$ are the left Haar measures of $K/L$ and $\widehat{K}/L^\perp$ respectively.
\end{theorem}
\begin{proof}
Continuity of the homomorphism $\tau^{\times,L}:H\to Aut(K/L\times\widehat{K}/L^\perp)$ guarantee that the semidirect product $G_{\tau^{\times,L}}=H\ltimes_{\tau^{\times,L}}\left(K/L\times\widehat{K}/L^\perp\right)$ is a locally compact group. Now, due to Theorem \ref{KL} and also Corollary \ref{TKL} for all $h\in H$ we have
\begin{align*}
d\sigma_{K/L}\times\sigma_{\widehat{K}/L^\perp}\left(\tau^{\times,L}_h(k+L,\omega L^\perp)\right)
&=d\sigma_{K/L}\times\sigma_{\widehat{K}/L^\perp}\left((k+L,\omega L^\perp)^h\right)
\\&=d\sigma_{K/L}\times\sigma_{\widehat{K}/L^\perp}(k^h+L,\omega_hL^\perp)
\\&=d\sigma_{K/L}(k^h+L)d\sigma_{\widehat{K}/L^\perp}(\omega_hL^\perp)
\\&=\delta_K(h)\delta_L(h^{-1})d\sigma_{K/L}(k+L)\delta_L(h)\delta_K(h^{-1})d\sigma_{\widehat{K}/L^\perp}(\omega L^\perp)
\\&=d\sigma_{K/L}(k+L)d\sigma_{\widehat{K}/L^\perp}(\omega L^\perp)
=d\sigma_{K/L}\times\sigma_{\widehat{K}/L^\perp}(k+L,\omega L^\perp).
\end{align*}
Hence, $d\mu_{G_{\tau^{\times,L}}}(h,k+L,\omega L^\perp)=dhd\sigma_{K/L}(k+L)d\sigma_{\widehat{K}/L^\perp}(\omega L^\perp)$ is the left Haar measure.
\end{proof}
Now let $L$ be a $\tau$-invariant uniform lattice in $K$. We define the $\tau$-Zak transform of $f\in L^2(G_\tau)$ via
\begin{equation}\label{ZAK}
\mathcal{Z}_Lf(h,k,\omega):=\delta_K(h)^{1/2}Zf_h(k^h,\omega_h)=\delta_K(h)^{1/2}\sum_{l\in L}f(h,k^h+l)\omega_h(l).
\end{equation}
When $f\in L^2(G_\tau)$ then for a.e. $h\in H$ we have $f_h\in L^2(K)$, so that $Z_Lf_h$ and hence $\mathcal{Z}_L$ is well-defined.
If $v\in L^2(K)$ and $u\in L^2(H)$ let $u\otimes v(h,k)=\delta_K(h)^{-1/2}u(h)v(k)$. Then, $u\otimes v\in L^2(G_\tau)$ and also we have
\begin{align*}
\mathcal{Z}_L(u\otimes v)(h,k,\omega)&=\delta_K(h)^{1/2}Z_Lf_h(k^h,\omega_h)
\\&=u(h)Z_Lv(k^h,\omega_h).
\end{align*}

The following proposition states some concrete properties of the $\tau$-Zak transform.
\begin{proposition}
{\it Let $f\in L^2(G_\tau)$. Then for a.e. $(h,k,\omega)\in G_{\tau^{\times,L}}$ and also $(l,\xi)\in L\times L^\perp$ we have
\begin{enumerate}
\item $\mathcal{Z}_Lf(h,k+l,\omega)=\overline{\omega_h(l)}\mathcal{Z}_Lf(h,k,\omega)$.
\item $\mathcal{Z}_Lf(h,k,\omega\xi)=\mathcal{Z}_Lf(h,k,\omega)$.
\end{enumerate}
}\end{proposition}
\begin{proof}
(1) Due to the quasi-periodicity of the Zak transform $Z_L$ we get
\begin{align*}
\mathcal{Z}_Lf(h,k+l,\omega)&=\delta_K(h)^{1/2}Z_Lf_h(k^h+l,\omega_h)
\\&=\delta_K(h)^{1/2}\overline{\omega_h(l)}Z_Lf_h(k^h,\omega_h)
=\overline{\omega_h(l)}\mathcal{Z}_Lf(h,k,\omega).
\end{align*}
(2) It is also guaranteed by the quasi-periodicity of the Zak transform $Z_L$.
\end{proof}
In the following theorem we show that the $\tau$-Zak transform defined in (\ref{ZAK}) is an isometric transform.
\begin{theorem}\label{PL}
Let $H$ be a locally compact group and $K$ be an LCA group also $\tau:H\to Aut(K)$ be a continuous homomorphism
and let $L$ be a $\tau$-invariant uniform lattice in $K$. The $\tau$-Zak transform
$\mathcal{Z}:L^2(G_\tau)\to L^2(G_{\tau^{\times,L}})$ is an isometric transform.
\end{theorem}
\begin{proof}
Using Theorem \ref{ZG} and also Fubini's theorem we have
\begin{align*}
\|\mathcal{Z}f\|_{L^2(G_{\tau^{\times,L}})}^2
&=\int_{G_{\tau^{\times,L}}}|\mathcal{Z}f(h,k,\omega)|^2d\mu_{G_{\tau^{\times,L}}}(h,k+L,\omega L^\perp)
\\&=\int_H\int_{K/L}\int_{\widehat{K}/L^\perp}|\mathcal{Z}f(h,k,\omega)|^2dhd\sigma_{K/L}(k+L)d\sigma_{\widehat{K}/L^\perp}(\omega L^\perp)
\\&=\int_H\left(\int_{K/L}\int_{\widehat{K}/L^\perp}|Zf_h(k^h,\omega_h)|^2d\sigma_{K/L}(k+L)d\sigma_{\widehat{K}/L^\perp}(\omega L^\perp)\right)\delta_K(h)dh
\\&=\int_H\left(\int_{K/L}\int_{\widehat{K}/L^\perp}|Zf_h(k,\omega)|^2d\sigma_{K/L}(k^{h^{-1}}+L)d\sigma_{\widehat{K}/L^\perp}(\omega_{h^{-1}} L^\perp)\right)\delta_K(h)dh
\\&=\int_H\|f_h\|^2_{L^2(K)}\delta_K(h)dh=\|f\|_{L^2(G_\tau)}^2.
\end{align*}
\end{proof}
Using linearity of the $\tau$-Zak transform and also the polarization identity we achieve the following orthogonality relation for the
$\tau$-Zak transform.
\begin{corollary}
{\it Let $H$ be a locally compact group and $K$ be an LCA group also $\tau:H\to Aut(K)$ be a continuous homomorphism
and let $L$ be a $\tau$-invariant uniform lattice in $K$. The $\tau$-Zak transform, for all $f,g\in L^2(G_\tau)$ satisfies the following orthogonality relation;
\begin{equation}
\langle\mathcal{Z}f,\mathcal{Z}g\rangle_{L^2(G_{\tau^{\times,L}})}=\langle f,g\rangle_{L^2(G_\tau)}.
\end{equation}}
\end{corollary}
%

\section{\bf{Examples and Applications}}

In this section via some examples we illustrate how the preceding results apply for semidirect product groups.

\subsection{Euclidean group}
Let $H=\mathrm{SL}(2,\mathbb{Z})$ and $K=\mathbb{R}^2$ also for all $\sigma\in H$ let $\tau_\sigma:\mathbb{R}^2\to\mathbb{R}^2$ be given via $\tau_\sigma(\mathbf{x})=\sigma\mathbf{x}$, for all $\mathbf{x}\in\mathbb{R}^2$. Then, $G_\tau=\mathrm{SL}(2,\mathbb{Z})\ltimes_\tau\mathbb{R}^2$ is a locally compact group with left Haar measure $d\mu_{G_\tau}(\sigma,\mathbf{x})=d\sigma d\bf{x}$. Then, $G_{\widehat{\tau}}=\mathrm{SL}(2,\mathbb{Z})\ltimes_{\widehat{\tau}}\widehat{\mathbb{R}^2}=\mathrm{SL}(2,\mathbb{Z})\ltimes_{\widehat{\tau}}{\mathbb{R}}^2$ where $\widehat{\tau}:\mathrm{SL}(2,\mathbb{Z})\to Aut(\widehat{\mathbb{R}}^2)$ is given by $\sigma\mapsto \widehat{\tau}_\sigma$ via $\widehat{\tau}_\sigma({\bf w})={\bf w}_\sigma$ for all ${\bf w}\in\widehat{\mathbb{R}}^2$, where for all $\mathbf{x}\in\mathbb{R}^2$ we have
\begin{align*}
\langle{\mathbf{x}},\widehat{\tau}_\sigma({\bf{w}})\rangle
&=\langle{\bf{x}},{\bf{w}}_\sigma\rangle
\\&=\langle\tau_{\sigma^{-1}}({\bf{x}}),{\bf{w}}\rangle
=\langle\sigma^{-1}{\bf{x}},{\bf{w}}\rangle
=e^{-2\pi i (\sigma^{-1}{\bf{x}},{\bf{w}})}=e^{-2\pi i (\bf{x}.\bf{w}\sigma^{-1})}.
\end{align*}
If $\sigma=\left(
             \begin{array}{cc}
               a & b \\
               c & d \\
             \end{array}
           \right)
\in\mathrm{SL}(2,\mathbb{Z})$, then for all $\mathbf{w}=(w_1,w_2)\in\widehat{\mathbb{R}^2}$ and ${\bf x}=(x_1,x_2)\in\mathbb{R}^2$ we have
\begin{equation}
\langle{\bf{x}},{\bf{w}}_\sigma\rangle=e^{-2\pi i (\sigma^{-1}{\bf{x}},{\bf{w}})}
=e^{-2\pi i(dw_1x_1-bw_1x_2-cw_2x_1+aw_2x_2)}.
\end{equation}
\begin{example}
Let $\alpha,\beta\in\mathbb{R}$ and $L_{\alpha,\beta}=\alpha\mathbb{Z}\times\beta\mathbb{Z}$. Then, $L_{(\alpha,\beta)}$ is a 2D $\tau$-invariant uniform lattice in $\mathbb{R}^2$ with $L_{(\alpha,\beta)}^\perp=L_{(\alpha^{-1},\beta^{-1})}$. The continuous homomorphism $\tau^{\times,L_{(\alpha,\beta)}}:\mathrm{SL}(2,\mathbb{Z})\to Aut(\mathbb{R}^2/L_{\alpha,\beta}\times\widehat{\mathbb{R}}^2/L_{(\alpha^{-1},\beta^{-1})})=Aut(\mathbb{T}^4)$ is given by $\sigma\mapsto\tau^{\times,L_{(\alpha,\beta)}}_\sigma$ via
\begin{equation}
\tau^{\times,L_{(\alpha,\beta)}}_\sigma(\mathbf{x}+L_{(\alpha,\beta)},\mathbf{w}+L_{(\alpha,\beta)}^\perp)
=(\sigma\mathbf{x}+L_{(\alpha,\beta)},\mathbf{w}\sigma^{{-1}}+L_{(\alpha,\beta)}^\perp).
\end{equation}
If $f:\mathrm{SL}(2,\mathbb{Z})\times\mathbb{R}^2\to\mathbb{C}$ satisfies
\begin{equation}
\int_{\mathrm{SL}(2,\mathbb{Z})}\int_{\mathbb{R}}|f(\sigma,\mathbf{x})|^2d\sigma d\mathbf{x}<\infty,
\end{equation}
then for $(\sigma,\mathbf{x},\mathbf{w})\in \mathrm{SL}(2,\mathbb{Z})\times \mathbb{R}^2\times\mathbb{R}^2$ we have
\begin{align*}
\mathcal{Z}_{L_{(\alpha,\beta)}}f(\sigma,\mathbf{x},\mathbf{w})
&=Z_{L_{(\alpha,\beta)}}f_\sigma(\sigma\mathbf{x},\mathbf{w}\sigma^{-1})
\\&=\sum_{n=-\infty}^{\infty}\sum_{m=-\infty}^{\infty}f(\sigma,\sigma\mathbf{x}+(\alpha n,\beta m))\langle(\alpha^{-1} n,\beta^{-1} m),\mathbf{w}\sigma^{-1}\rangle.
\end{align*}
\end{example}

\subsection{Weyl-Heisenberg group}
Let $K$ be an LCA group with the Haar measure $dk$ and $\widehat{K}$ be the dual group of $K$ with the Haar measure $d\omega$
also $\mathbb{T}$ be the circle group and let the continuous homomorphism $\tau:K\to Aut(\widehat{K}\times\mathbb{T})$ via $s\mapsto \tau_s$ be given by $\tau_s(\omega,z)=(\omega,z.\omega(s))$. The semidirect product $G_\tau=K\ltimes_\tau(\widehat{K}\times\mathbb{T})$ is called the Weyl-Heisenberg group associated with $K$ which is also denoted by $\mathbb{H}(K)$. The group operation for $(k,\omega,z),(k',\omega',z')\in K\ltimes_\tau(\widehat{K}\times\mathbb{T})$ is
\begin{equation}
(k,\omega,z)\ltimes_\tau(k',\omega',z')=(k+k',\omega\omega',zz'\omega'(k)).
\end{equation}
If $dz$ is the Haar measure of the circle group, then $dkd\omega dz$ is a Haar measure for the Weyl-Heisenberg group and also the continuous homomorphism $\delta:K\to (0,\infty)$ given in (\ref{t}) is the constant function $1$. Thus, using Theorem 4.5 and also Proposition 4.6 of \cite{FollH} and Proposition \ref{T} we can obtain the continuous homomorphism $\widehat{\tau}:K\to Aut(K\times\mathbb{Z})$
via $s\mapsto\widehat{\tau}_s$, where $\widehat{\tau}_s$ is given by
$\widehat{\tau}_s(k,n)=(k,n)\circ \tau_{s^{-1}}$ for all $(k,n)\in K\times\mathbb{Z}$ and $s\in K$.
Due to Theorem 4.5 of \cite{FollH}, for each $(k,n)\in K\times \mathbb{Z}$ and also for all $(\omega,z)\in \widehat{K}\times\mathbb{T}$ we have
\begin{align*}
\langle(\omega,z),(k,n)_s\rangle&=\langle(\omega,z),\widehat{\tau}_s(k,n)\rangle
\\&=\langle\tau_{s^{-1}}(\omega,z),(k,n)\rangle
\\&=\langle(\omega,z\overline{\omega(s)}),(k,n)\rangle
\\&=\langle\omega,k\rangle\langle z\overline{\omega(s)},n\rangle
\\&=\omega(k)z^n\overline{\omega(s)}^n
\\&=\omega(k-ns)z^n
=\langle\omega,k-ns\rangle\langle z,n\rangle=\langle(\omega,z),(k-ns,n)\rangle.
\end{align*}
Thus, for all $k,s\in K$ and $n\in\mathbb{Z}$ we have
\begin{equation}\label{000}
\widehat{\tau}_s(k,n)=(k,n)_s=(k-ns,n).
\end{equation}
Therefore, $G_{\widehat{\tau}}$ has the underlying set $K\times K\times\mathbb{Z}$ with the following group operation;
\begin{align*}
(s,k,n)\ltimes_{\widehat{\tau}}(s',k',n')&=\left(s+s',(k,n)\widehat{\tau}_s(k',n')\right)
\\&=\left(s+s',(k,n)(k'-n's,n')\right)=(s+s',k+k'-n's,n+n').
\end{align*}
\begin{example}
Let $K=\mathbb{Z}$ then $G_{{\tau}}=\mathbb{Z}\ltimes_{{\tau}}\mathbb{T}^2$, where $\tau:\mathbb{Z}\to Aut(\mathbb{T}^2)$ is defined via $\ell\mapsto \tau_\ell$ given by $\tau_\ell(w,z)=(w,zw^\ell)$ for all $w,z\in\mathbb{T}$ and $\ell\in\mathbb{Z}$.
Due to (\ref{000}), $\widehat{\tau}:\mathbb{Z}\to Aut(\mathbb{Z}^2)$ is given by $\ell\mapsto \widehat{\tau}_{\ell}$ where $\widehat{\tau}_\ell(p,q)=(p-q\ell,q)$ for all $\ell,p,q\in\mathbb{Z}$.
Now for all $n\in\mathbb{Z}$ let $\mathbb{T}_n=\{z\in\mathbb{T}:z^n=1\}$ and also for $(n,m)\in\mathbb{Z}^2$ let  $L_{(n,m)}=\mathbb{T}_n\times\mathbb{T}_m$. Then,
\begin{align*}
L_{(n,m)}^\perp&=(\mathbb{T}_n\times\mathbb{T}_m)^\perp
\\&=\mathbb{T}_n^\perp\times\mathbb{T}_m^\perp=n\mathbb{Z}\times m\mathbb{Z}.
\end{align*}
If $n|m$ then $L_{(n,m)}$ is a $\tau$-invariant uniform lattice in $\mathbb{T}^2$. In this case
$\tau^{\times,L_{n,m}}:\mathbb{Z}\to Aut(\mathbb{T}^2/L_{(n,m)}\times\mathbb{Z}^2/L_{(n,m)}^\perp)$ for all $\ell\in\mathbb{Z}$ is given by
\begin{align*}
\tau^{\times,L_{n,m}}_\ell\left((w,z)+L_{(n,m)},(p,q)+L_{(n,m)}^\perp\right)
&=\left(\tau_\ell^{L_{(n,m)}}(w,z)+L_{(n,m)},\tau_\ell^{L_{(n,m)}^\perp}(p,q)+L_{(n,m)}^\perp\right)
\\&=\left((w,z)^\ell+L_{(n,m)},(p,q)_\ell+L_{(n,m)}^\perp\right)
\\&=\left((w,zw^\ell)+L_{(n,m)},(p-q\ell,q)+L_{(n,m)}^\perp\right).
\end{align*}
Since for all $n\in\mathbb{Z}$ we have $\mathbb{T}/\mathbb{T}_n=\mathbb{T}$ and also $\mathbb{Z}/n\mathbb{Z}=\mathbb{Z}_n$ we can consider
\begin{equation}
\tau^{\times,L_{n,m}}:\mathbb{Z}\to Aut(\mathbb{T}^2\times\mathbb{Z}_n\times\mathbb{Z}_m)
\end{equation}
via $\tau^{\times,L_{n,m}}_\ell(w,z,p,q)=\left(w,zw^\ell,p-q\ell,q\right)$ for all $(w,z,p,q)\in\mathbb{T}^2\times\mathbb{Z}_n\times\mathbb{Z}_m$.
If $f:\mathbb{Z}\times\mathbb{T}^2\to\mathbb{C}$ satisfies
\begin{equation}
\sum_{\ell=-\infty}^{\infty}\int_{0}^{2\pi}\int_{0}^{2\pi}|f(\ell,e^{i\theta},e^{it})|^2dtd\theta<\infty,
\end{equation}
then for $(\ell,e^{i\theta},e^{it},p,q)\in \mathbb{Z}\times\mathbb{T}^2\times\mathbb{Z}^2$ we have
\begin{align*}
\mathcal{Z}_{L_{(n,m)}}f(e^{i\theta},e^{it},p,q)
&=Z_{L_{(n,m)}}f_\ell\left(e^{i\theta},e^{it},p,q\right)
\\&=Z_{L_{(n,m)}}f_\ell\left((e^{i\theta},e^{i(t+\theta\ell)}),(p-q\ell,q)\right)
\\&=\sum_{(w,z)\in L_{(n,m)}}f(\ell,w+e^{i\theta},z+e^{i(t+\theta\ell)})(p-q\ell,q)(w,z)
\\&=\sum_{(w,z)\in L_{(n,m)}}f(\ell,w+e^{i\theta},z+e^{i(t+\theta\ell)})w^{p-q\ell}z^q
\\&=\sum_{k=1}^n\sum_{j=1}^mf(\ell,e^{i(2\pi k/n+\theta)},e^{i(2\pi j/m+t+\theta\ell)})e^{\frac{2\pi ik}{n}(p-q\ell)}e^{\frac{2\pi ij}{m}q}.
\end{align*}
According to Theorem \ref{PL} we get
\begin{align*}
&\sum_{\ell=-\infty}^{\infty}\int_0^{2\pi}\int_0^{2\pi}\sum_{p=0}^n\sum_{q=0}^m\left|\sum_{k=1}^n\sum_{j=1}^mf(\ell,e^{i(2\pi k/n+\theta)},e^{i(2\pi j/m+t+\theta\ell)})e^{\frac{2\pi ik}{n}(p-q\ell)}e^{\frac{2\pi ij}{m}q}\right|^2dtd\theta
\\&\ =\sum_{\ell=-\infty}^{\infty}\int_0^{2\pi}\int_0^{2\pi}\sum_{p=0}^n\sum_{q=0}^m\left|\mathcal{Z}_{L_{(n,m)}}f(\ell,e^{i\theta},e^{it},p,q)\right|^2dtd\theta
\\&\ =\sum_{\ell=-\infty}^{\infty}\int_{0}^{2\pi}\int_{0}^{2\pi}|f(\ell,e^{i\theta},e^{it})|^2dtd\theta.
\end{align*}
\end{example}

{\bf ACKNOWLEDGEMENTS.}
The authors would like to gratefully acknowledge financial support from the Numerical Harmonic Analysis Group (NuHAG) 
at the Faculty of Mathematics, University of Vienna.

\bibliographystyle{amsplain}

\end{document}